\documentclass[reqno]{amsart}
\usepackage{amsaddr}
\usepackage[english]{babel}
\usepackage[latin1]{inputenc}
\usepackage{amsfonts}
\usepackage{amssymb}
\usepackage{amsmath}
\usepackage{amsthm}
\usepackage{graphicx}
\usepackage{epsfig}
\usepackage{fancyhdr}
\usepackage[all]{xy}
\usepackage{amsmath, amsthm, amsfonts, amssymb, mathrsfs,mathtools,faktor}
\usepackage{graphicx}
\usepackage{enumerate}
\usepackage{hyperref}
\usepackage{amsmath,amssymb}
\usepackage[bbgreekl]{mathbbol}
\numberwithin{equation}{section}

\newcommand{\C}{\mathbb{C}}
\newcommand{\R}{\mathbb{R}}
\newcommand{\N}{\mathbb{N}}
\newcommand{\Z}{\mathbb{Z}}

\newtheorem{theorem}{Theorem}[section]
\newtheorem{lemma}{Lemma}[section]

\newtheorem{obs}{Observation}[section]

\newtheorem{prop}{Proposition}[section]
\newtheorem{defi}{Definition}[section]

\newcommand{\ip}[2]
      {\ensuremath{\langle #1,#2 \rangle}}

\begin{document}
	
\title{ Remling's Theorem for vector-valued discrete Schr\"odinger operators } 
 \author[Acharya]{Keshav Raj Acharya}

\address{Department of Mathematics\\
Embry-Riddle Aeronautical University\\  
1 Aerospace Blvd., Daytona Beach, FL 32114, U.S.A. \\ E-mail address:acharyak@erau.edu }

\begin{abstract} This paper extends Remling's Theorem to vector-valued discrete Schr\"odinger operators, showing that the $\omega$ limit points of the matrix potentials, under the shift map, are reflectionless on the absolutely continuous spectrum with full multiplicity.\\


\end{abstract}
\maketitle

\section{Introduction}

We consider discrete Schr\"odinger equations with matrix-valued potentials of the form
\begin{align}\label{ds}
y(n+1)+y(n-1)+B(n)y(n)=zy(n),\quad z\in \C 
\end{align}
where \( y : I \to \mathbb{C}^d \) is a vector-valued sequence defined on an index set \( I \subseteq \mathbb{Z} \), typically \( \mathbb{Z} \) (the full lattice) or \( \mathbb{N} \) (the half-line). Each $ y: I \rightarrow \C^d: y(n) =(y_1(n)\quad y_2(n), \hdots\quad y_d(n))^{ \top } $ ( $ \top $ stands for a transpose)  is a vector-valued sequence, and \( B(n) \in \mathbb{C}^{d \times d} \) is a Hermitian matrix referred to as the matrix-valued potential.\\
Solutions to equation \eqref{ds} are obtained by prescribing the values of the sequence at two consecutive sites. That is, for any \( c_1, c_2 \in \mathbb{C}^d \), one defines \( y(n) = c_1 \), \( y(n+1) = c_2 \), and the rest of the sequence is determined recursively via \eqref{ds}. Since each vector has \( d \) components, the total dimension of the solution space is \( 2d \).

A matrix-valued sequence \( F(n) \in \mathbb{C}^{d \times 2d} \) is called a \emph{matrix solution} of equation \eqref{ds} if each column of \( F(n) \) satisfies \eqref{ds}. Equivalently, \( F(n) \) satisfies the recurrence relation component-wise. Equation \eqref{ds} induces a linear operator on the Hilbert space \( \ell^2(I, \mathbb{C}^d) \), consisting of square-summable vector-valued sequences. The inner product on this space is defined by
\[
\langle u, v \rangle = \sum_{n \in I} u(n)^* v(n),
\]
where \( u(n)^* \) denotes the Hermitian transpose (conjugate transpose) of \( u(n) \). Define an operator \( J \) on \( \ell^2(I, \mathbb{C}^d) \) by
\begin{equation}\label{jo}
	(Jy)(n)= y(n+1)+y(n-1)+B(n)y(n) 
	. \end{equation}	 
When \( I \) is semi-infinite, such as \( I = \mathbb{N} \), one must modify the definition at the boundary. For instance, set
\[
(Jy)(1) = y(2) + B(1) y(1), \]
by imposing \( y(0) = 0 \) as the boundary condition at 0. Assuming that the potential \( B(n) \) is uniformly bounded and Hermitian for all \( n \), the operator \( J \) is self-adjoint. By the spectral theorem, its spectrum is a closed subset of \( \mathbb{R} \) and can be decomposed into absolutely continuous, singular continuous, and pure point parts.

This matrix-valued generalization is particularly significant in the analysis of physical systems with internal degrees of freedom such as coupled waveguides, spin chains, or multichannel quantum models where scalar equations fail to capture the full spectral complexity.
 
Remling's theorem, originally formulated for one-dimensional scalar Schr\"odinger and Jacobi operators \cite{CR1, CR2}, describes the asymptotic structure of potentials in terms of their spectral properties. Specifically, it asserts that all \(\omega\)-limit points of a potential sequence under the shift map are reflectionless on the essential support of the absolutely continuous spectrum.

In this paper, we extend Remling's theorem to vector-valued discrete Schr\"odinger operators. We consider a bounded, real, symmetric  matrix-valued potentials \( B(n) \) on \( \mathbb{N} \), and show that the corresponding operator \( J \) inherits an analogous reflectionless property: every \(\omega\)-limit point of the sequence \( B(n) \) under the shift map gives rise to a reflectionless operator on the absolutely continuous spectrum, with full multiplicity. Similar extensions in the context of canonical systems and CMV matrices can be bound in \cite{KA1,db}.   
This generalization not only broadens the scope of Remling's original result, but also highlights the stability of the reflectionless property in higher-dimensional and matrix-valued settings. It confirms that the absolutely continuous spectrum retains strong structural features even in the presence of internal degrees of freedom.

\section{Titchmarsh-Weyl $m$ function}

The Titchmarsh-Weyl theory provides elegant methods to study the spectral measure and the corresponding spectrum of the associated operator $J$. Some of the results presented here can also be found in \cite{KA2, KM}, though we present them here to make the paper self-contained.  In \cite{KA2}, it is shown that for any bounded Hermitian matrix-valued potential $B(n)$  and for any $z\in \C\setminus\R$, there is a precise number $d$ of linearly independent solutions of \eqref{ds} that lie in space $\ell^2(\N, \C^d)$. Having these solutions in columns, we can construct a matrix solution $F(n)$ that lies in the space $\ell^2(\N, \C^{d\times d} )$, the space of all  $d\times d $  square-summable matrix-valued sequences $U(n)$.
That is, \[
\ell^2(\N, \C^{d\times d} ) = \Big \{ U(n) : \sum_{n=1}^{\infty}\|U(n)^* U(n)\| < \infty  \Big \}.
\]
The Wronskian of any two vector-valued sequences $\{x(n)\}$ and $\{y(n)\}$, is defined by
\begin{equation}
W_n(x(n), y(n)) = x(n+1)^\top y(n) - x(n)^\top y(n+1).
\end{equation} This definition is generalized to the $d \times d$ matrix-valued sequences $\{F(n)\}$ and $\{G(n)\}$ as 
\begin{equation}
W_n(F(n), G(n)) =   F(n+1)^\top G(n) - F(n)^\top G(n+1).
\end{equation}
 In \cite{KM}, it is shown that if $\{F(n)\}$ and $\{G(n)\}$ are matrix solutions to \eqref{ds} then the Wronskian $W_n(F(n), G(n))$ is independent of $n.$\\
Define the difference expression
\[
\tau(u(n)) = u(n + 1) + u(n - 1) + B(n)u(n), 
\]  
which resembles the operator \( J \) but acts on general vector-valued sequences.
\medskip

\begin{lemma}[Green's Identity] Let $\mathbb{N}_0 = \mathbb{N} \cup \{0\}$. For $F(n, z), G(n, z) \in \ell^2(\mathbb{N}_0, \mathbb{C}^{d\times d} )$, the following identity holds:
\[
\sum_{j=1}^{n} \left[ F(j, z)^* \, (\tau G)(j, z) - (\tau F)(j, z)^* \, G(j, z) \right]
= W_0(\overline{F}, G) - W_n(\overline{F}, G).
\]
\end{lemma}
\begin{proof}   By using the definition of $\tau$ we expand:
\begin{align*}
F(j)^* (\tau G)(j) - (\tau F)(j)^* G(j) 
&= \left[ F(j)^* G(j+1) - F(j+1)^* G(j) \right] \\
&\quad + \left[ F(j)^* G(j-1) - F(j-1)^* G(j) \right] \\
&\quad + F(j)^* (B(j) - B(j)^*) G(j).
\end{align*} 
The last term  on the right-hand side is  zero since $B(j)$ is Hermitian. Summing from $j=1$ to $n$ on both sides we get,
\begin{align*}
\sum_{j=1}^{n}  [ F(j, z)^* \, & (\tau G)(j, z) -  (\tau F)(j, z)^* \, G(j, z)  ] \\ & = F(1)^*G(0) -F(0)^*G(1) +F(n)^*G(n+1) -F(n+1)^*G(n) )
 \\ & = W_0(\overline{F}, G) - W_n(\overline{F}, G).
\end{align*}
\end{proof}
As mentioned above the solution space of \eqref{ds} is a $2d$-dimensional vector space. Suppose a basis for the space  is of the form  $ \{ u_1,  u_2,  \dots, u_d, v_1, v_2, \dots, v_d \} $. Using this basis, we form matrix solutions, that is, for $z\in \C,$ let 
\begin{align}\label{is} 
U(n,z) = (u_1(n,z), u_2(n,z), \hdots, u_d(n,z)), \ \  V(n,z) = (v_1(n,z), v_2(n,z), \hdots, v_d(n,z)) .
\end{align} 
In addition, we prescribe the initial conditions
\begin{equation}
\label{ic} U(0,z)= -I, \hspace{.3in} V(0,z)= 0, \hspace{.2in}U(1,z) =0, \hspace{.3in} V(1,z)= I .
\end{equation}

We may choose these conditions at different locations. More specifically, when we define the Titchmarsh-Weyl $m$ functions on the half line, these conditions are assigned at any point $n.$ By iterating the difference equation \eqref{ds}, we see that for fixed $n\in \N,\  U(n,z),  V(n,z)$ are polynomials in $z$ of degree $n-2$ over $ \C^{d\times d}.$ In addition, if $ B(n)$ is real, the complex conjugates of the matrix solutions for   spectral parameters $z$ are same as the matrix solutions for the complex conjugates parameters $\bar{z}.$ That is, 
$$
\overline{U(n,z)} = U(n,\bar{z})\ , \quad \overline{V(n,z)}=V(n,\bar{z}).$$

The Titchmarsh-Weyl $m$ function for the vector-valued discrete Schr\"odinger operators associated to the equation (\ref{ds}) is defined in terms of solutions as follows. 

\begin{defi} Let $z\in \C^+= \{z\in \C : \operatorname{Im}(z) >0\}.$ The Titchmarsh-Weyl $m$ function is defined as the unique complex matrix $ M(z) \in \C^{d\times d}$ such that
\begin{align} \label{wm}
F(n,z) = U(n,z)+ M(z)V(n,z) \in\ell^2(\N, \C^{d\times d} ),
\end{align} where $U(n,z), V(n,z) $ are matrix solutions with initial values (\ref{ic}).\end{defi}

 Note that the matrix  solution $  F(n,z)$ is a set of $d$ linearly independent vector solutions of  (\ref{ds}) that are in $l^2(\N, \C^d).$ This definition is  well defined. As we mentioned above that there are only $d$ linearly independent solutions in $l^2(\N, \C^d)$, if there is another $M(z)$ satisfying the above conditions then the solutions $v_j$ from  $V(n,z)$ will be in $l^2(\N , \C^d)$. The solution   $V(n,z)$ is such that $V(0,z)= 0$, which implies that $V(n,z)$ is the set of eigen-functions for a self adjoint operator $J.$ This contradicts that the spectrum of $J$ is a set of real numbers.
 The Titchmarsh-Weyl $m$ function defined on \eqref{wm} can be expressed in terms of  matrix solutions and in terms of resolvent operators as  shown in the following theorem from \cite{KA2}. 

\begin{theorem}  Let $z\in \C^+.$ If $(\tau-z)F= 0$  and $F $ is a $d\times d$ matrix-valued solution whose $d$ columns are linearly independent solutions of (\ref{ds}) that are in $ l^2(\N, \C^d).$ Then \begin{align} 
\label{smf} M(z)= -F(1,z)F(0,z)^{-1} .
\end{align} Moreover, 
\begin{align} \label
{mf2}M(z) =(m_{ij}(z))_{d\times d} \in \C^{d\times d}, \quad m_{ij}(z) =\ip{\bbdelta_j}{(J-z)^{-1}\bbdelta_i}.
\end{align} 
\end{theorem}

The $ M(z)$ in \eqref{smf} is obtained from \eqref{wm} by plugging $ n =0$ and $ n=1$ and is expressed in the form for aligning the definitions of $m$-functions in \eqref{wmf} and \eqref{mf}.

In \eqref{mf2}, the Dirac delta type vector-valued sequences   $\bbdelta_j(n)$ are defined as 
\begin{align*} 
	&\bbdelta_1 = \left \{\begin{bmatrix}1\\0\\  \vdots \\ 0 \end{bmatrix}, \begin{bmatrix}0\\0\\  \vdots \\ 0 \end{bmatrix}, \dots  \right\}, \bbdelta_2 = \left \{\begin{bmatrix}0\\1\\  \vdots \\ 0 \end{bmatrix}, \begin{bmatrix}0\\0\\  \vdots \\ 0 \end{bmatrix},  \dots  \right\},\ldots,
\bbdelta_d = \left \{ \begin{bmatrix}0\\0\\  \vdots \\ 1 \end{bmatrix},  \begin{bmatrix}0\\0\\  \vdots \\ 0 \end{bmatrix},  \dots \right\}
.\end{align*}
 By functional calculus, for each $ i, j$ there exists positive Borel measure  $\mu_{ij}$ such that
 \begin{align} \label{ir}
 m_{ij}(z) =\ip{\bbdelta_j}{(J-z)^{-1}\bbdelta_i} = \int_{\R} \frac{1}{t-z} d\mu_{ij}.
 \end{align}
 Let $ \mu =  \begin{pmatrix}  & \mu_{11} \    &\mu_{12} \ \dots &\mu_{1d} \\   &\mu_{21} \ & \mu_{22} \dots &\mu_{2d}     \\   & \vdots & \vdots  &\vdots \\  &\mu_{d1} &\mu_{d2} \dots & \mu_{dd}\end{pmatrix}. $ Then $\mu$ is a matrix-valued measure and the Weyl $m$ function takes the following integral representation 
  \begin{align}
  \label{irm} M(z) =  \int_{\R} \frac{1}{t-z} d\mu   \ .
  \end{align} 
 The matrix-valued measure $\mu$ is a \textit{spectral measure} of the  operator $J.$  The imaginary part of $M(z)$,     $ \operatorname {Im}M(z) = \frac{1}{2i}(M(z)- M(z)^*)  $ satisfies an important property shown in the following proposition.
	
	\begin{prop} \cite{KA2} For $z\in \C^+,$ the Weyl $m$ function $M(z)$ from \eqref{wm} is symmetric and has positive definite imaginary part. That is,  $$M(z)^*= M(\bar{z}), \quad \operatorname {Im}M(z)> 0.$$ \end{prop}	
	
	\begin{proof}   For each $i, j$,  the entry $m_{i,j}(z)$ maps the complex upper half plane to itself. For if $ z \in \C^+$, by \eqref{ir}, we get
    $$ 
    \operatorname{Im} m_{ij}(z)= \frac{1}{2i}(  m_{ij}(z) - m_{ij}( \bar{z})) = \int_{\R} \frac{\operatorname{Im} z }{|t-z|^2}d\mu_{ij}  > 0 .
    $$
 Suppose $\overline{M(z)}$ denotes the complex conjugate of $M(z)$ obtained by taking the complex conjugate of each entry of $ M(z) $. Then by (\ref{ir}), i.e., an integral representation of $ m_{ij}(z)$, we have $ \overline{ m_{ij}(z) } = m_{ij}( \bar{z}))$ so that   $ \overline{M(z)} = M(\bar{z})$. By using the self-adjointness of $J$, for each $i,j $	we get,	
	\begin{align*} 
    m_{ij}(z) = & \ip{\bbdelta_j }{(J-z)^{-1}\bbdelta_i} \\ = & \ip{ ( (J-z)^{-1})^* \bbdelta_j}{\bbdelta_i} \\ = &  \ip{ (J-\bar{z})^{-1}\bbdelta_j}{\bbdelta_i} \\  
 = & \overline{\ip{\bbdelta_i}{(J-\bar{z})^{-1}\bbdelta_j} }\\	 
			= & \overline{m_{ji}(\bar{z})}\\
			= & m_{ji}(z) .
            \end{align*}
It follows that $M(z)$ is symmetric. Next, we show that $\operatorname{Im} M(z) >0 .$ Replace $G(n,z)$ by $F(n,z)$ from \eqref{wm}  in the Green's Identity we obtain,
   \begin{align} \label{GI}
\sum_{j=1}^{n}  [ F(j, z)^* \, & (\tau F)(j, z) -  (\tau F)(j, z)^* \, F(j, z)  ] = W_0(\overline{F}, F) - W_n(\overline{F}, F).
\end{align} 
Since $F(n,z) \in \ell^2(\N, \C^{d\times d} ) $ is a matrix solution to \eqref{ds} $$\tau F =zF, \  \text{and}  \displaystyle \lim_{n\rightarrow \infty} W_n( \overline F, F)= O .$$
Taking the limit as $ n\rightarrow \infty $ in \eqref{GI} yields	 \begin{align} \label{eq 11}
(z- \overline{z})\sum_{j=1}^{\infty}    F(j, z)^*  F(j, z)   = W_0(\overline{F}, F).
\end{align} 	
 As  $ W_0(\overline{F}, F) =  F^{*}(1, z)F(0,z)-  F^{*}(0,z) F(1,z) $,  using the initial conditions \eqref{ic} for $F(n,z) = U(n,z) +V(n,z)M(z)$, we get $F(1,z) = M(z)$, \ $F(0,z)= -I$ and so $$  W_0(\overline{F}, F) = M(z) -M(z)^*$$
 It follows that the \eqref{eq 11} becomes
\begin{align} \label{eq 12}
\operatorname{Im} (z)\sum_{j=1}^{\infty}    F(j, z)^*  F(j, z)   = \operatorname{Im}M(z).
\end{align}
Since the left-hand side converges to a positive definite matrix, $\operatorname{Im}M(z)$ is positive definite.

\end{proof}

This proposition tells us that $M(z)$ is a matrix-valued Herglotz function. By the matrix-valued Herglotz function $M: \C^+ \rightarrow M_d(\C)=\C^{d\times d}$ we mean that $M(z)$ is analytic and $\operatorname{Im} (M(z)) >  0$ for $z\in \C^+.$ By the Herglotz representation theorem, $  M(z)$ can also have the following integral representation, 
\begin{align}
\label{eq 2.8}  M(z) = A + Bz + \int_{\R}\Big(\ \frac{1}{t-z}-\frac{t}{t^2+1}\Big)d \mu(t), \quad  z\in \C^+
\end{align}
for some positive matrix-valued Borel measure $\mu$
on $ \R $ with $\int\frac{1}{t^2+1}d\mu_{i,j} < \infty$ and matrices $ A , B \in\R^{d \times d },\, B \geq 0.$ The Borel measures obtained in \eqref{irm} and in \eqref{eq 2.8} are equivalent. This integral representation of $M$ is a well-known result, see \cite{SF, FG} for more theory about these functions.  In addition,  $M(z)$ has finite normal limits, that is,
\begin{equation*}
	M(t\pm i0) = \lim_{\varepsilon\downarrow 0} M(t\pm i\varepsilon)
\end{equation*}
for almost all $t\in\R$. \\

 Recall that a Borel measure $\rho$ on $\R$ is called {\em absolutely continuous } if  $ \rho(B)=0$ for all Borel sets $B\subset\R$ of Lebesgue measure zero. By the Radon-Nikodym Theorem, $\rho$ is absolutely continuous if and only if $d\rho= f(t)dt$ for some density $f \in L_{loc}^1(\R), \quad f\geq 0.$ If $\rho$ is supported by a Lebesgue null set that is, there exists a Borel set $B\subset\R$ with $ |B|=\rho(B^c)=0$, then we say that $\rho$ is {\em singular}. Here, $|.|$ denotes the Lebesgue measure. By Lebesgue's decomposition theorem, a positive Borel measure $\rho$ can be decomposed as a sum of absolutely continuous $\rho_{ac}$, singular continuous $\rho_{sc}$, and pure point $\rho_{pp}$ measures
 \begin{equation} 
 \label{ld}	\rho = \rho_{ac} + \rho_{sc} +\rho_{pp}. 
\end{equation}
A set $S_{\rho} \subset \R$ is called a support of $\rho$ if $\rho(\R \setminus S_{\rho}) = 0$ and a support $S_{\rho}$ is called minimal relative to the measure $\nu$ if $ A \subset S_{\rho} $ with $\rho (A) = 0$ implies $\nu(A)=0.$ The minimal supports are unique up to sets of $\rho$ and $\nu$ measure zero. We will call a support of $\rho$ is ``minimal" if the support is minimal relative to the Lebesgue measure. The minimal support is also called an essential support.

The matrix-valued spectral measure $\mu$ on $\R$ from \eqref{eq 2.8} can be decomposed, in the same way as in \eqref{ld}, into absolutely continuous $\mu_{ac}$, singular continuous $\mu_{sc}$  and pure point $\mu_{pp}$ as 
\begin{equation} \label{ldm}
\mu = \mu_{ac} + \mu_{sc} +\mu_{pp}.
\end{equation} 
In \cite{FG}, these measures are expressed in terms of corresponding Weyl $m$ function, a  matrix-valued Herglotz function $M(z)$ with positive definite imaginary part, $ \operatorname{Im} M(z) >0 $.  The absolutely continuous part is \begin{equation*}
    d\mu_{ac} = \pi^{-1} \operatorname{ Im} M(t +i 0 )dt . 
    \end{equation*} Let   $\Sigma_{\mu}$, $\Sigma_{ac}$, $ \Sigma_{sc}$, and $\Sigma_{pp}$ be the supports of $\mu$, $\mu_{ac}$, $\mu_{sc}$, and $\mu_{pp}$ respectively.  The following support theorem from \cite{FG} establishes the connection between these supports with  $ M(z) .$ Since $\operatorname{Im} M(z) >0, $ $M(z)$ is invetible, therefore the support in the following theorem are minimal. We state it without proof.
\begin{theorem} \cite{FG} The essential (minimal) supports of the measures in \eqref{ldm} are given by \begin{align*}
			& \Sigma_{ac}  = \bigcup_{r=1}^d \left\{t\in\R : \lim_{\varepsilon\downarrow 0} M_+(t + i\varepsilon)\textrm{ exists, } \operatorname{rank }\operatorname{Im}(M_+(t + i0)) = r \right\} .\\ &
			\Sigma_{sc}   = \bigcup_{r=1}^d \left\{t\in\R : \lim_{\varepsilon\downarrow 0} \operatorname{Im}(\operatorname{tr } M_+(t + i\varepsilon))=\infty,\,\lim_{\varepsilon\downarrow 0}\varepsilon\operatorname{tr }M_+(t + i\varepsilon)=0 \right\}. \\ &	 
			\Sigma_{pp}   = \bigcup_{r=1}^d \left\{t\in\R : \operatorname{rank }\lim_{\varepsilon\downarrow 0}\varepsilon M_+(t + i\varepsilon) = r \right\}\,.
		\end{align*}
	 
\end{theorem}
  Moreover, the essential support of $\mu_{ac}$  with full multiplicity is given by
\[ 
\Sigma_{ac}^d = \{ t \in \R :  \operatorname{rank ( Im} M(t +i0) ) = d \quad a. e. \text{  in a neighborhood of }  t \}.
\] It follows that $t \in \Sigma_{ac}^d $ if and only if $ \operatorname{ Im}M(t +i0) ) >0 $ almost everywhere.

\section{Topologies on the Space of Potentials}
In this section, we introduce a topology on the space of matrix-valued potentials. While Hermitian potentials were considered in the preceding sections, and several important properties were established, the main theorem will be proved only for real, symmetric, and uniformly bounded potentials. Let 
 $M_d(\R)$ denote the space of all $d \times d $ real matrices, equipped with the topology induced by the operator norm.
   Let $\mathfrak B_C $ denote the set of all potentials $\mathbf B = \{B(n): n\in \Z \}$ from \eqref{ds}. A potential $\mathbf B \in \mathfrak B_C$ is such that for each $n$,  $B(n)$ is real, symmetric, and bounded in the operator norm $ \|B(n)\|  = \sup_{\|x\| =1} \|B(n)x \|. $  In addition, $\mathbf B $
 is uniformly bounded: $  \sup_{n \in \Z} \|B(n) \| < \infty  .$
So, $$
\mathfrak B_C = \{ \mathbf B = B(n):   \text{ for all } n,   B(n) \in   M_d (\R),  B(n)^\top = B(n), \ \|B(n)\|\leq C    \}.
$$
 
To construct a topology on the space $\mathfrak B_C $, suppose that $\mathfrak B_C^i =\{B(i): \mathbf B \in \mathfrak B_C \},$ is the collection of $i^{th}$ component matrix  of each matrix sequence $ \mathbf B $. Then $\mathfrak B_C^i$ is a subspace of $M_d(\R)$.   Moreover, since $\mathfrak B_C^i$ is a closed and bounded subspace of $M_d(\R)$, $\mathfrak B_C^i$ is a compact topological space with the subspace topology. We can consider $\mathfrak B_C $ as a product space $$\mathfrak B_C = \prod _{i\in \Z} \mathfrak B_C^i  .$$ In fact, $\mathfrak B_C$ is a compact topological space with respect to the product topology. Consider a metric on $\mathfrak B_C$ by 

\begin{equation} \label{metric} 
d(\mathbf A, \mathbf B)=  \sum_{n=-\infty}^{\infty} 2^{-|n|}\|A(n)-B(n)\|.
\end{equation}
We refer to a matrix sequence $B(n)$ as a half-line sequence if it is defined on $\N $ (or $- \N $ ) or on a similar interval in $\Z$. If $B(n)$ is a half line potential, we can extend it to a whole line by considering $B(n) = 0 $ for all $n \not \in \N$
 ( or $n \not \in - \N$).  We call two matrices $A$ and $B$ equivalent if $ A = U B U^{-1}$ where $U$ is a unitary matrix. For a potential $ \mathbf B, $ a shift map $S$  is defined by  $$ (S^k \mathbf B)(n))=B(n+k).$$

\begin{defi} Suppose 	$B(n)$  in \eqref{ds} is a half line potential. The $  \omega $ limit set of  	$B(n)$ under the shift map is defined as,
	\begin{equation} \label{w} 
    \omega( \mathbf B)= \{ H(n)\in \mathfrak B_C : \exists \ n_j \rightarrow \infty \text{ such that } d(S^{n_j}B, H)\rightarrow 0\}. \end{equation}
\end{defi} 
        The limits $H(n)$ in \eqref{w} are whole-line potentials. These $\omega$ limit sets have the following important properties.
            
\begin{prop} $\omega (\mathbf B) $ is a non-empty compact subset of $ \mathfrak B_C $ and the shift map $S$ is homeomorphism on $\omega (B) .$ Moreover, 
\begin{equation}\label{3.3}
d(S^{n}\mathbf B, \omega(\mathbf B))\rightarrow 0, n \rightarrow \infty . \end{equation}
\end{prop}

\begin{proof} Extend the half line potential $\mathbf B$ to a whole line potential by considering $B(n) = 0 $ for $n \leq 0 .$ Then  $\omega(\mathbf B)$ can be expressed as \begin{align*} \omega (\mathbf B) =   \bigcap_{m\geq1} \overline{ \{ S^n \mathbf B : n\geq m\} }.\end{align*}
	 This shows that $\omega(\mathbf B)$ is an intersection of a decreasing sequence of compact sets. So, $\omega(\mathbf B)$ is non-empty and compact. In addition, $\omega(\mathbf B)$ is invariant under $S$ and $S^{-1}, $ so $S$ is   homeomorphism on  $\omega(\mathbf B)$.  
	 
	 Moreover, if there is some $\epsilon > 0$ such that $d(S^{n_j} \mathbf B, \mathbf H) \geq \epsilon $ for all $j$ and for all   $ \mathbf H \in \omega(\mathbf B)$ then there is no subsequence $n_j$ for which  $d(S^{n_j} \mathbf B, \mathbf H)\rightarrow 0 $ which is a contradiction, hence \eqref{3.3} holds. \end{proof}

\section{ Main theorems and their proofs } \label{mainth}
  For any $z\in \C^+, $ define the Weyl $m$ functions on $ \mathbb I_- = \{ \dots, n-2, n-1, n\}$ and $ \mathbb I_+ =\{n+1, n+2,  \dots\}$ by	
\begin{align} \label{wmf}
	M_+(n,z)  = - F_+(n+1,z)  F_+(n,z)^{-1},\\ \nonumber  M_-(n,z)  = - F_-(n-1,z)  F_-(n,z)^{-1},
\end{align}
where $F_{\pm}(n,z)$ are matrix  solutions to \eqref{ds} whose columns form   linearly independent vector solutions to the equation and  $F_{\pm}(n,z)\in\ell^2(\mathbb I _{\pm},\C^{d \times d}).$ These solutions are unique up to a multiplication by a constant matrix and are invertible for  each $n$. For $n=0, \ \	M_+(0,z)  = M(z)$ in \eqref{wm}. Thus, $M_+(n,z)$ is a matrix coefficient  as   the one in \eqref{wm} with the condition \eqref{ic} at $n$. Therefore $M_+(n,z)$ enjoy the similar properties as $ M(z)$.    Also, the same is true to $M_-(n,z) $  as it is a mirror version of $M_+(n,z)$. Hence, for any $n$ and any $z \in \C^+, $  $M_{\pm}(n, z) $ are symmetric and  have positive definite imaginary part.   Thus as a function of $z\in \C^+$, these are matrix-valued Hertgloz functions  mapping the complex upper half plane to Siegel upper half plane, the  space of the complex symmetric matrices with positive definite imaginary part: 
\begin{equation*}
	\mathcal{S}_d = \{Z\in\C^{d \times d} : Z = X +iY,\, X^\top=X,\, Y^\top = Y,\, Y> 0 \}\,.
\end{equation*}
Let $ \mathbb H$ denotes the set of all  matrix-valued Herglotz functions. For each matrix potential $B(n)$ in \eqref{ds}, there exists a unique $m$ function $M^B \in \mathbb H .$
Since for each $n$,   $ M_{\pm}(n, z)$ are Hertgloz functions, 
 the normal limit 
\begin{equation*}
	M_{\pm}( n, t) = \lim_{\varepsilon\downarrow 0} M_{\pm}(n, t \pm i\varepsilon)
\end{equation*} 
exists for almost every $t.$ Let  $\mathbf B_{\pm}$ denote the restrictions of   potentials $\mathbf B $  on     $   \Z_+ = \{1, 2, 3, \hdots \}$ and on $  \Z_-=\{ \hdots, -2, -1, 0 \}$ respectively. Suppose
$$
\mathfrak B_C^{\pm} =\{ \mathbf B_{\pm}: \mathbf B\in \mathfrak B_C \} .
$$  
 So the associated $m$ functions with the initial conditions at $n=0$,  
 $$\{ M_{\pm} =  M_{\pm}^ \mathbf B (0, z) \} \subset \mathcal S_d .$$
 
 Let  $ M_{\pm}(t) = M_{\pm}(0, t) .$  The space of m-functions $ \mathbb H$ is   equipped with a topology of uniform convergence or a topology related to the space of analytic functions, where small changes in the m-function correspond to small changes in the underlying spectral properties of the matrix. Similarly the subspaces $\mathfrak B_C^{\pm}$ are also topological spaces with the topology induced by the metric defined in \eqref{metric}.
 
 \begin{prop}  \cite{MM}\label{prop 4.1} The maps  $ \Psi_{\pm} : \mathfrak B_C^{\pm}  \mapsto \mathcal S_d, \ \  \Psi ( \mathbf B_{\pm}) =  M_{\pm}$ where $  M_{\pm} = M_{\pm}^\mathbf B (0,z) $  are homeomorphism onto their images. \end{prop}
 \begin{proof} In \cite{MM}, it has been proved that there is a one-to-one correspondence between block Jacobi matrices ( or $B_{\pm}$ ) and a Hermitian measure, and a unique Weyl-$m$ function $M_{\pm}$ up to a constant multiple. Also the map $ \Psi_{\pm} $ and its inverse are continuous.     \end{proof}

 Then we state the definition of reflectionless potentials.
\
\begin{defi} Let $A\subset\R$ be a Borel set. We call a potential $ B(n) $ from equation (\ref{ds}) on $\Z$, reflectionless on $A$ if
	\begin{align}\label{ref}
    M_{+}(t)=-\overline{M_{-}(t)},
    \end{align}
    for almost every  $ t\in A$. Here  $M_{\pm}(z)$ are Weyl m functions associated with equation (\ref{ds}) restricted to $\Z_{\pm}.$
\end{defi}
The set of all reflectionless potentials on $A$ is denoted by $\mathcal R(A).$

\begin{theorem}
	\label{RT} Let $\mathbf B $  be a (half line)
	potential, and let $\Sigma_{ac}^d$ be the essential support of the 	absolutely continuous part of the spectral measure with full multiplicity. Then \[
	\omega(\mathbf  B)\subset \mathcal R ( \Sigma_{ac}^d).\] \end{theorem} 
This theorem is an extension of the Remling's theorem for Jacobi equation from one dimensional space, see \cite{CR1}. This theorem is a consequence of the  Breimesser-Pearson Theorem. Therefore in order to prove the Remling's theorem, we first extend  the Breimesser-Pearson Theorem for the equation \eqref{ds}.\\

The Breimesser-Pearson theorem provides a characterization of the spectrum of these operators, particularly focusing on the conditions under which the spectrum is purely continuous and the absence of eigenvalues. Please refer to the articles \cite{BP1, BP2, Pearson} about this theorem.

First, recall the harmonic measure on the complex upper half plane, for any $z\in \C^+$ and any Borel set $S\subset \R$,  \begin{equation} \label{eq 4.3}\omega_z(S) = \frac{1}{\pi} \int_S \frac{1}{(t-x)^2 +y^2} dt. \end{equation} Notice that $ \omega_z(S)$ satisfies 
\begin{equation} \label{hmp}
\omega_z(S) = \omega_{-\overline{z}}(-S). 
\end{equation} 
Please see \cite{CR2} for more detail about these measures. We  want to define the harmonic measure for a matrix-valued Herglotz functions. For any scalar Herglotz function  
$g$ and $t\in \R $ we define $ \omega_{g(t)}(S) = \lim_{y \rightarrow 0^+} \omega_{g(t+iy)}(S).$ 

\begin{obs} Let  $G(z)$ be a matrix-valued Herglotz function. For any $c \in \C^d$ the map $z \mapsto c^*G(z)c$ is a scalar Herglotz function, and   $$ c^*G(t)c = \lim_{y \rightarrow 0^+} c^*G(t+iy)c.$$ \end{obs} Consequently, for any Borel set $S\subset \R$ we have
$$  \omega_{c^*G(t)c}(S) = \lim_{y \rightarrow 0^+} \omega_{c^*G(t+iy)c}(S).$$ 
Let $\mathbb B(0)= \{ z \in \C : |z| \leq 1 \}$ be the closed unit ball in $\C$.

\begin{defi} Let $ G_n, G \in \mathbb H $ be matrix-valued Herglotz functions. Then we say that

\begin{itemize}

\item $G_n$ converges uniformly to $ G$ on a set $ S \subset \C^+$ if 
$$ \lim_{n\rightarrow \infty} \sup _{z \in S} \| G_n(z) -G(z)\| = 0 ,$$ where $ \|.\|$ is a matrix norm.

\item $G_n$ converges  to $ G$ in value distribution if for any $c\in \mathbb B(0)$ 
\begin{equation}
\lim_{n\rightarrow \infty }  \int_A \omega_{c^*G_n(t)c}(S) dt = \int_A \omega_{c^*G(t)c}(S) dt ,
\end{equation} for all Borel set $A, S \subset \R ,\ \ |A| < \infty .$
\end{itemize}

\end{defi}
The following theorem provides the equivalence statement about the uniform convergence and convergence in value distribution. This is a matrix version of the Thereom 2.1 from \cite{CR1}.

\begin{theorem}\label{thm 4.6}  Suppose $ G_n, G \in \mathbb H$ be matrix-valued Herglotz functions. Then the following are equivalent: 
	\begin{enumerate}   \item $G_n \rightarrow G$ uniformly on compact subsets of $\C^+.$
 \item $G_n \rightarrow G$ in value distribution.\end{enumerate}\end{theorem}

\begin{proof} Assume that \( G_n \to G \) uniformly on compact subsets of \( C^+ \). This means that for every compact set \( K \subset C^+ \),
\[
\sup_{z \in K} \| G_n(z) - G(z) \| \to 0 \quad \text{as} \quad n \to \infty.
\]
We need to show that \( G_n \to G \) in value distribution. This requires proving that for all Borel sets \( A, S \subset \mathbb{R} \) with \( |A| < \infty \) and for any $c\in \mathbb{B} (0)$
\[
\lim_{n \to \infty} \int_A \omega_{c^*G_n c}(S) \, dt = \int_A \omega_{c^*G c }(S) \, dt. \]
where \( \omega_{z}(S) \) denotes the harmonic measure associated with \( z \) and \( S \).

Since \( G_n \to G \) uniformly on compact subsets of \( C^+ \), the measures \( \mu_n \)  converge weakly to \( \mu \) where \( \mu_n \) and  \( \mu \) are the measures in the integral representation of \( G_n \) and \( G \) respectively,  from \eqref{eq 2.8}. Therefore, by the properties of weak convergence of measures, we have, for any $c\in \mathbb{B} (0)$
\[
\lim_{n \to \infty} \int_A \omega_{c^*G_n c}(S) \, dt = \int_A \omega_{c^*G c }(S) \, dt. \]
This establishes the convergence in value distribution.

Assume that \( G_n \to G \) in value distribution. This means that for all Borel sets \( A, S \subset \mathbb{R} \) with \( |A| < \infty \), and for any $c\in \mathbb{B} (0)$
\[
\lim_{n \to \infty} \int_A \omega_{c^*G_n c}(S) \, dt = \int_A \omega_{c^*G c }(S) \, dt. \]
We need to show that \( G_n \to G \) uniformly on compact subsets of \( C^+ \).  For any $c \in \C^d$ the map $z \mapsto c^*G(z)c$ is a scalar Herglotz function, likewise so is $c^*G_n(z)c$. Then we      use the fact that convergence in value distribution implies uniform convergence on compact subsets of \( C^+ \) for   scalar Herglotz functions.

By the theorem on convergence of scalar Herglotz functions, convergence in value distribution implies uniform convergence on compact subsets of \( C^+ \). Therefore, \( G_n \to G \) uniformly on compact subsets of \( C^+ \).
 \end{proof}

\begin{lemma} \label{lemma 4.6} Let $A \subset \R$ be a Borel set with $|A| < \infty $. Then  for any $c\in \mathbb B(0)$ 
$$ \lim_{y\rightarrow 0+} \sup_{M\in \mathbb H, \ \ S \subset \R} \Big| \int_A \omega_{c^*M(t+iy)c}(S) dt  - \int_A \omega_{c^*M(t)c}(S) dt \Big| = 0 .$$\end{lemma}
\begin{proof} This follows from the fact that for any $c\in \mathbb B(0), $ 
 the map $ z\mapsto c^*M(t+iy)c$ is a scalar Herglotz function and using the Lemma A.1 from \cite{CR1} or Theorem 1 from \cite{BP2}.\end{proof}
 Let $ \N_- = \{1,\ldots , n\}$ and  on $\N_+ =\{ n+1, n+2,\ldots\}$ and  take the initial condition at $n$. That is, the matrix solutions $U$ and $V$ satisfy the following initial conditions at $n$: 	 
\begin{equation*}
	\begin{aligned}
		&U(n)  =  - I, && V(n) = 0 \\ 
		&U(n+1)= 0,  && V(n+1) = I\,.
	\end{aligned}
\end{equation*}
For any $z\in \C^+, $ define the Weyl $m$ functions on $\N_-$ and $\N_+$ by	
\begin{equation}\label{ss}
	F_{\pm}(n,z) = U(n,z) \pm  \tilde M_{\pm}(n,z) V(n,z),
\end{equation}
where we require that $F_+\in\ell^2(\N,\C^{d \times d})$ and we take $F_-(0) = 0$ for normalization of the matrix  solutions to equation \eqref{ds}. Notice that the columns of $F_{\pm}(n,z)$ form a linearly independent set of solutions to equation (\ref{ds}) and therefore for each $n$, $F_{\pm}(n)$ is invertible.  For $z\in\C^+$, these half-line $m$ functions are uniquely determined and are given by:
\begin{equation} \label{mf}
	\tilde M_+(n,z) = - F_+(n+1,z)F_+(n,z)^{-1},\ \  \tilde M_-(n,z)  =  F_-(n+1,z)F_-(n,z)^{-1}\,.
\end{equation}
These $m$ functions are related to the one defined in \eqref{wmf}  by the following relation $$M_+(N,z) = \tilde M_+(N,z), \ \ M_{-}(N,z) = B(n) -zI + \tilde  M_{-}(N,z) .$$ 

\begin{theorem} \label{theorem 4.8}  Let $\Sigma_{ac}^d$ denotes the essential support with full multiplicity of absolutely continuous part of the spectral measure of the half-line problem \eqref{ds}. Then for any $A \subset \Sigma_{ac}^d, \ |A| < \infty $ and $S \subset \R,$ we have 
\begin{equation*}
\lim_{N\rightarrow \infty}\Big ( \int_A \omega_{c^* \tilde M_-(N,t)c} (-S) dt -\int_A \omega_{c^* \tilde M_+(N,t)c}(S) dt\ \Big ) = 0,
\end{equation*} 
for all $c\in \mathbb B(0).$ Moreover, the convergence is uniform in $S .$\end{theorem}

To prove this theorem, we will follow similar techniques as in one-dimensional case, see \cite{CR1}, and as for the canonical system, see \cite{KA1}. However, the $m$ functions are matrix-valued Herglotz functions and are the matrices in $\mathcal S_d .$ The space $\mathcal S_d$ can be considered as a generalization of the complex upper half plane.   

We also want to consider a metric on  $\mathcal S_d$ as a generalization of the hyperbolic metric 
\begin{equation}\label{hm}
	ds = \frac{\sqrt{dx^2 +dy^2}}{y}
\end{equation}
on the complex hyperbolic plane $\C^+.$ Define the following   map on $\mathcal{S}_d$, $$d_\infty:\mathcal{S}_d\times\mathcal{S}_d\to\R$$ via	
\begin{equation}\label{ms}
	d_{\infty}(Z_1, Z_2) = \inf_{Z(t)} \int _0^1 F_{Z(t)}(\dot{Z}(t))\,dt\,, \quad Z_1,  Z_2\in\mathcal{S}_d
\end{equation}
where
\begin{equation}\label{norm} 
	F_Z(W) = \|Y^{-1/2}WY^{-1/2}\|,
\end{equation} 
called the Finsler norm, and the infimum is taken over all differentiable paths $Z(t)$ joining $Z_1$ to $Z_2$. Here $Y$ is a positive definite matrix and possesses the square root and $Y^{-1/2} = (Y^{1/2})^{-1} = (Y^{-1})^{1/2}$.  The norm in equation (\ref{norm}) is the operator norm of matrices acting on $\C^d$. A calculation shows that $d_\infty$ is a metric on $\mathcal{S}_d$ and is called the Finsler metric. Hence $(\mathcal{S}_d, d_\infty)$ is a metric space. In the case when $d=1$, the Siegel upper half space $\mathcal S_1 = \C^+$ and the metric $d_{\infty}$ is same as the the hyperbolic metric defined in equation (\ref{hm}) on $\C^+$.  The length of a curve in the Euclidean plane is measured by $\sqrt{dx^2 +dy^2}$.   Similarly, the length of a curve $\gamma(t) = x(t) +i y(t)$, $t\in [0,1]$ on the hyperbolic plane $\C^+$ is defined by
\begin{equation*}
	h(\gamma) = \int _0^1 \frac{1}{y(t)}\sqrt{\left(\frac{dx}{dt}\right)^2 + \left(\frac{dy}{dt}\right)^2}\ dt.
\end{equation*}
The hyperbolic distance between two points $z, w \in C^+$ is defined by 

\begin{equation*}
	\rho(z,w) = \inf h(\gamma)
\end{equation*}
where the infimum is taken over all piecewise differentiable curves joining $z$ and $w$.  So for  $d=1$,  we have $d_{\infty}(z, w)= \rho (z, w)$.  Thus the Finsler metric $d_{\infty}$ on $\mathcal S_d$ is a generalization of the hyperbolic metric $\rho$ on $\C^+$. Another notion of distance on $\C^+$ is a pseudo hyperbolic distance which is closely connected with harmonic measure. As in \cite{CR1}, the pseudohyperbolic distance of two points $z, w \in \C^+$ is defined by  \[ \gamma(w,z)
= \frac{|w-z|}{\sqrt{ \text{Im}w}\sqrt{\text{Im}z}}.\] 

The hyperbolic metric $\rho $ and the pseudohyperbolic distance $\gamma$ are related by the following equation (see \cite{BP2})
\begin{equation}\gamma(w,z) =2 \sinh(\rho(z, w)) .\end{equation}

In \cite {CR1}, it is shown that hyperbolic distance and harmonic measure satisfy,
\begin{equation}\label{eq 4.10}|\omega_{w}(S)-\omega_{z}(S)|\leq \gamma(w,z)\end{equation} for any $z,w \in \C^+ $ and any Borel set $S\subset\R.$ 
 From \eqref{eq 4.10}, for any matrix-valued Herglotz functions $ M_1 , M_2 \in \mathbb H$ we have
 \begin{equation} \label{eq 4.11}
 |\omega_{c^* M_1 c}(S)-\omega_{c^*M_2c}(S)|\leq \gamma(c^*M_1c,c^*M_2 c),
 \end{equation} for any $c \in \mathbb B(0)$ and any Borel set $S\subset\R.$ 
 
 One important observation about the matrices is the following inequality
 
 \begin{equation} \label{eq 4.11'}
 d_{\infty}(c^*M_1c,c^*M_2 c) \leq d_{\infty}(M_1, M_2 ),
 \end{equation} 
 for any $c \in \mathbb B(0)$ and any   $ M_1, M_2 \in \mathbb H.$

Another important tool that we need here is the fractional linear transformation.  For any $S\in M_{2d }(\C)$ of the form:
\begin{equation*}
	S= \begin{pmatrix} A & B \\ C & D \end{pmatrix}
\end{equation*}
where are $A$, $B$, $C$, and $D$ are $d\times d$ matrices, a {\it matrix-valued fractional   transformation} is a map $S: M_{d }(\C) \rightarrow  M_{d }(\C)$ defined by
\begin{equation}\label{flt}
	S(Z) = (AZ + B) (CZ +D)^{-1}\,,
\end{equation}
for all $Z \in M_{d }(\C) $ for which $S$ is well defined. Let $ J= \begin{pmatrix}0 & -I \\ I & 0 \end{pmatrix}.$  In \cite{KM}, it is shown that the  restriction of the map \eqref{flt} on $\mathcal S_d$ is well-defined if $S$ satisfies $i(S^*JS-J) \geq 0 .$  We now extract such matrices from \eqref{ds} that act as a matrix-valued fractional transformation.
Suppose $F(n)$ is a matrix solution to the equation (\ref{ds}). Then we have 
\begin{equation}\label{tm}
	\begin{bmatrix} F(n+1) \\ \mp F(n) \end{bmatrix} = \begin{pmatrix} z I -B(n) & \pm I \\ \mp I  & 0 \end{pmatrix} \begin{bmatrix} F(n) \\ \mp F(n-1) \end{bmatrix}\,
.\end{equation}
   The matrices 
\begin{equation}\label{tm1} 
	T_{\pm}(n,z) = \begin{pmatrix} z I -B(n) & \pm I \\ \mp I  & 0 \end{pmatrix}
\end{equation}
given in equation (\ref{tm}) are called {\it transfer matrices} which describe the evolution of the vectors
\begin{equation*}
	\begin{bmatrix} F(n+1) \\ \mp F(n) \end{bmatrix}
\end{equation*}
under iteration of $T_{\pm}(n,z)$. Notice that  
     
 \begin{equation}\label{tm2} \begin{pmatrix} I & 0 \\ 0 & -I \end{pmatrix}	T_{+}(n,z)  = T_{-}(n,z)   \begin{pmatrix} I & 0 \\ 0 & -I \end{pmatrix}.\end{equation}

These transfer matrices $T_{\pm}(n,z)$ can be considered as complex matrix-valued fractional transformations \eqref{flt} acting on the space of $m$ functions as follows:
\begin{equation*}
	T_{\pm}(n,z)\tilde M_{\pm}(n,z) = \mp \Big( (z I -B(n)) - \tilde M_{\pm}(n,z))^{-1} \Big).
\end{equation*}
If $B(n)$ is real and symmetric, these transfer matrices satisfy the following symplectic identity:
\begin{lemma} The transfer matrices $ T_{\pm}(n, z) $ satisfy the following symplectic identity: $$ T_{\pm}(n, z)^\top J T_{\pm}(n,z ) = J .$$ \end{lemma}
\begin{proof} We verify for $T_+(n,z)$
\begin{align*} T_+(n,z)^\top J T_+(n,z)  & = \begin{pmatrix} z I -B(n) &   I \\ - I  & 0 \end{pmatrix}^\top \begin{pmatrix} 0 & I \\ -I & 0 \end{pmatrix} \begin{pmatrix} z I -B(n) &   I \\ - I  & 0 \end{pmatrix}  \\ & =
\begin{pmatrix} z I -B(n) &   -I \\  I  & 0 \end{pmatrix} \begin{pmatrix}-I &   0 \\  -zI+B(n)  & -I \end{pmatrix} \\ & = J\end{align*} 
Similarly, it holds for $T_-(n,z).$
\end{proof}
Also if $z\in \R,\ \ T_{\pm}(n, z)$ are real symplectic matrices which means that  $ T_{\pm}(n, z) \in \operatorname{Aut}{(\mathcal S_d)}$ meaning that they are bijective holomorphic functions, and moreover, the action is distance preserving with respect to the metric \eqref{ms}, see \cite{KM1}. That is, for $ T = T_{\pm}(n, z) \in SL(2d, \R) $ and  $W_1, W_2 \in \mathcal S_d $ we have 
 		\begin{align} \label{dp2} d_{\infty} \Big(T( W_1), T( W_2 )\Big) = d_{\infty}( W_1, W_2) \end{align}   

and for $z\in \C^+, \ \ T_{+}(n, z) $ does not map $\mathcal S_d  $ to itself in general.
In \cite{KM}, it is shown that $$ \tilde M_{\pm}(n,z) = 	T_{\pm}(n,z)\tilde M_{\pm}(n-1,z) ,$$ which implies $$\tilde M_+ (n,z) =T_{+}(n,z)T_{+}(n-1,z) \dots T_{+}(1,z) \tilde M_+(0,z) .$$  and $$ \tilde M_- (n,z) =T_{-}(n,z)T_{-}(n-1,z) \dots T_{-}(2,z) \tilde M_-(1,z) .$$
 Introduce $$P_{+}(n,z):= T_{+}(n,z)T_{\pm}(n-1,z) \dots T_{+}(1,z).$$ and $$ P_-(n,z) := T_{-}(n,z)T_{-}(n-1,z) \dots T_{-}(2,z).$$ Thus  \begin{equation} \label{eq 4.14}  \tilde  M_+ (n,z) =P_{+}(n,z) \tilde M_+(0,z), \    \ \tilde M_- (n,z) =P_{-}(n,z) \tilde M_-(1,z)    
 .\end{equation}
\begin{lemma} \label{lemma4.9}The matrices $P_{+}(n,z) $ and $P_{-}(n,z)$ satisfy the following relation
\begin{equation}\begin{pmatrix} I & 0 \\ 0 & -I \end{pmatrix}	P_{+}(n,z)  = P_{-}(n,z)  \begin{pmatrix} I & 0 \\ 0 & -I \end{pmatrix} T_{+}(1,z)  .\end{equation}\end{lemma}

\begin{proof} It follows from \eqref{tm2}\end{proof}
 We note that for any real symmetric potential $B(n)$ in \eqref{ds} and $z\in \C^+,$ the transfer matrix $ T_{-}(n,z) $ can be written as     
    \begin{align*} T_{-}(n,z) =  \begin{pmatrix}  I  &   -B(n) \\ 0 &  I  \end{pmatrix} \begin{pmatrix} I &  zI \\ 0  & I \end{pmatrix}  \begin{pmatrix} 0 & -I \\ I  & 0 \end{pmatrix} = T_B  T_z T_J  .\end{align*} The matrices $ T_J,  T_B    \in SL(2d, \R) $, so $ T_J(Z), T_B(W) \in \mathcal S_d $ for any $ Z,  W \in \mathcal S_d .$ We now show that $T_z(Z) \in \mathcal S_d$ for any $ Z\in \mathcal S_d$. But this is also clear because $T_z (Z) =Z +zI $ which is in $\mathcal S_d .$ The following lemma from \cite{KM} shows the distance decreasing property of the action of $T_-(n,z)$ on $\mathcal S_d$. For completeness of the paper, we present the proof.

\begin{lemma} For any $n\in \N, \ \ z =x+iy \in \C^+$ and $ W_1, W_2 \in \mathcal S_d,$
\begin{equation}\label{di}
d_{\infty}(T_-(n, z)W_1, T_-(n, z)W_2 )  \leq  \frac{1}{1+y^2}	d_{\infty}( W_1,  W_2) .\end{equation}
\end{lemma}
\begin{proof}  Write $T_{-}(n,z) =  T_B  T_z T_J$. Since $T_B,  T_J  \in SL(2d, \R) $, by \eqref{dp2}, we get 
	\begin{align} \label{2.271} \nonumber d_{\infty} \Big(T_{-}(n,z)( W_1), T_{-}(n,z)( W_2 )\Big) & = d_{\infty} \Big(T_B  T_z T_J( W_1), T_B  T_z T_J( W_2 )\Big) \\ & = d_{\infty} \Big( T_z T_J( W_1), T_z T_J( W_2 )\Big). \end{align} Let $U_j = T_J(W_j) = JW_j$, then $T_z(U_j) = U_j +zI, \, j =1,2.$  Let $U(t)$ be a length minimizing path between $U_1$ and $U_2.$ Then $U(t)+zI$ be a path between $U_1 +zI$ and  $U_2 +zI$, then $\frac{d}{dt} (U(t) +zI) = \dot{U(t)}.$ Suppose $V(t)= \operatorname{Im} U(t)$ then $\operatorname{Im}\big(U(t) +zI \big) = V(t) + yI.$ Now
\begin{align*}
F_{(U(t) +zI)} ( \dot{U(t)}) & =  \| \Big(V(t) + yI\Big)^{-\frac{1}{2}} \dot{U(t)} \Big(V(t) + yI\Big)^{-\frac{1}{2}}\| \\ & = \| \Big(V(t) + yI \Big)^{-\frac{1}{2}} V(t) ^{\frac{1}{2}}V(t)^{-\frac{1}{2}}\dot{U(t)} V(t)^{-\frac{1}{2}}V(t)^{\frac{1}{2}} \Big(V(t) + yI \Big)^{-\frac{1}{2}}\|\\ & \leq \| (V(t) + yI)^{-\frac{1}{2}} V(t)^{\frac{1}{2}} \| \| V(t)^{-\frac{1}{2}} \dot{U(t)} V(t)^{-\frac{1}{2}} \|  \| V(t)^{\frac{1}{2}} ( V(t) + yI)^{-\frac{1}{2}} \| \end{align*}
Since $ V(t)^{\frac{1}{2}}, ( V(t) + yI)^{-\frac{1}{2}}$ are symmetric $$\| (V(t) + yI)^{-\frac{1}{2}} V(t)^{\frac{1}{2}} \| = \| V(t)^{\frac{1}{2}} ( V(t) + yI)^{-\frac{1}{2}} \|.$$
It follows that 
\begin{equation} \label{eq 4.23}  F_{(U(t) +zI)} ( \dot{U(t)}) \leq \| (V(t) + yI)^{-\frac{1}{2}} V(t)^{\frac{1}{2}} \|^2 \| V(t)^{-\frac{1}{2}} \dot{U(t)} V(t)^{-\frac{1}{2}} \|  .\end{equation}
Next we show that $ \| (V(t) + yI)^{-\frac{1}{2}} V(t)^{\frac{1}{2}} \|^2 \leq \frac{1}{1+y}.$
Since $U_j = JW_j$, $U(t)$ must be of the form $JW(t)$ where $W(t)$ is a path between $W_1$ and $W_2.$ Again since $ W_j = T_- (0, z)Z_j$, $W$ is of the form $  W = T_- (0, z)Z $ where $Z(t)= X(t)+iY(t)$ is a path between $Z_1$ and $Z_2.$ So \begin{align*}
W(t) & =T_- (0, z)Z(t) \\ & = \begin{pmatrix} z I -B(0) &  I \\ - I  & 0 \end{pmatrix} Z(t) \\ & = zI -B(0)- Z(t)^{-1} \\ & = zI -B(0)- \overline{Z(t)} |Z(t)|^{-1},\quad |Z| = X^2 +Y^2  .\end{align*}
Thus, $ \operatorname{Im} W(t) = yI + Y(X^2+Y^2)^{-1} > yI .$
 Taking the inverse and applying operator theory, we get $$ ( \operatorname{Im} W(t))^{-1}< \frac{1}{y}I.$$ Since $U(t) = JW(t)= -W(t)^{-1}, $ $$ U(t) = -W(t)^{-1} = - \overline{W(t)} |W(t)|^{-1} = - \overline{W(t)} \Big(  ( \operatorname{Re} W(t))^2 + (\operatorname{Im} W(t))^2\Big)^{-1}.$$ Therefore, $ \operatorname{Im} U(t) =  \operatorname{Im} W(t) \Big(  ( \operatorname{Re} W(t))^2 + (\operatorname{Im} W(t))^2\Big)^{-1}.$  Since $ \operatorname{Im} W(t) \leq \Big( \operatorname{Re} W(t))^2 + (\operatorname{Im} W(t))^2\Big)\operatorname{Im} W(t)^{-1},$ and taking the inverse we obtain,
$$ \operatorname{Im} W(t) \Big(  ( \operatorname{Re} W(t))^2 + (\operatorname{Im} W(t))^2\Big)^{-1}  \leq \operatorname{Im} W(t)^{-1} < \frac{1}{y}I. $$ This implies that
$ \operatorname{Im} U(t)< \frac{1}{y} I $, that is $ V(t) < \frac{1}{y} I$.
By Lemma 3.7 from \cite{KM1}, we get  $$ \| (V+yI)^{- \frac{1}{2}} V^{\frac{1}{2}}\| < \frac{1}{1+y^2}.$$
Eqn \eqref{eq 4.23} becomes 

\begin{align}\label{2.28}
F_{(U(t) +zI)} ( \dot{U(t)}) & \leq \| (V(t) + yI)^{-\frac{1}{2}} V(t)^{\frac{1}{2}} \|^2 \| V(t)^{-\frac{1}{2}} \dot{U(t)} V(t)^{-\frac{1}{2}} \| & \leq \frac{1}{1+y^2} F_{(U(t))} ( \dot{U(t)}) .\end{align}
On integration  Eqn \eqref{2.28} becomes
$$ \int_0^1 F_{(U(t)+z)} ( \dot{U(t)}) dt \leq  \frac{1}{1+y^2} \int_0^1 F_{(U(t))} ( \dot{U(t)}) dt
.$$ Taking the infimum over all such paths $U(t) $ we get

\begin{align} \label {2.29}d_{\infty} \Big(T_{z} U_1, T_{z}  U_2 \Big) \leq \frac{1}{(1+ y^2) } d_{\infty}( U_1, U_2). \end{align}
Since $J \in SP(2d., \R), $  \begin{align} \label{2.30} d_{\infty}( U_1, U_2) = d_{\infty}( T_JW_1, T_JW_2) = d_{\infty}( W_1, W_2) .\end{align}
Using \eqref{2.29} and \eqref{2.30} in \eqref{2.271} we get the desired result
\begin{align*}  d_{\infty} \Big(T_{-}(n,z) W_1, T_{-} (n,z) W_2 \Big) \leq \frac{1}{(1+ y^2) } d_{\infty}( W_1, W_2). \end{align*}
\end{proof}

\begin{lemma} \label{lemma 4.11 } Let K be a compact subset of $\C^+.$ Then for any $W\in S_d,$
	 $$ \lim_{n\rightarrow \infty}	d_{\infty}(\tilde M_-(n, z),  P_-(n, z)W) = 0  $$  uniformly in $z\in K .$\end{lemma}
 
 \begin{proof}  By writing $\tilde M_-(n, z) = P_{-}(n,z) \tilde M_-(1,z),$ it follows from \eqref{di}  that $$		d_{\infty}(\tilde M_-(n, z),  P_-(n, z)W) \leq  \frac{1}{(1+y^2)^{n-1}}	d_{\infty}( \tilde M_-(1, z),  W)  .$$ Thus by taking the limit as $n \rightarrow \infty $ we get the result. \end{proof}

\begin{lemma}  Then for any $M\in S_d,$
	 $$ -\overline{P_+(n,t)M} = P_-(n,t)(-T_+(1, t)\overline{M} )   .$$\end{lemma}
 \begin{proof} It follows from Lemma \ref{lemma4.9}.  \end{proof}

\begin{proof}[Proof of Theorem \ref{theorem 4.8}] Let $ A \subset \Sigma_{ac}^d$ with $|A| < \infty$ and let $\epsilon >0$ be given. Partition   $A = A_0\cup A_1 \cup \dots \cup A_N$ of disjoint 
 subsets such that $|A_0| <\epsilon,  $ and $ A_j$ is bounded for $j\geq 1.$ We require that $M_+(t) = \lim_{y \rightarrow 0^+} M_+(t+iy)$ exists and $M_+(t) \in \mathcal S_d$ on $ \bigcup_{j=1}^{N} A_j .$ To find $A_j$ with these properties, first of all put $t \in A$ for which $M_+(t)$ does not exists or does not lie in $\mathcal S_d $ into $A_0.$ Clearly, $|A_0|\leq \epsilon.$ Then pick a sufficiently large compact subset $S \subset \mathcal S_d$ that contains the range of $M_+= M_+(0,t)$, and a compact subset $ t\in K \subset \R .$ Subdivide $S$ into finitely many disjoint subsets   $ S_1, S_2,  \dots ,S_n$  such that for all $j=1, \dots, n, $ we have that, for any $ Z_1, Z_2 \in S_j,$   $$ d_{\infty}(Z_1, Z_2) \leq \epsilon .$$ Then define $A_j = (A \setminus A_0)\cap \tilde M_+^{-1}(S_j)$ and $ M^j = \tilde M_+(t_j) $ for any fixed $ t_j \in A_j. $ Note that for each $j$, $\tilde M_+(t_j)$ is invertible. These $ M^j $ satisfy the  inequality:  \begin{equation} \label{eq 4.16} d_{\infty}(\tilde M_+(t), \ M^j) \leq \epsilon \end{equation} for all $t \in A_j,\quad j\geq 1 .$
 
 By Lemma \ref{lemma 4.6}, for any $c\in \mathbb B(0),$ there is a number $y>0$ such that, for arbitrary matrix-valued Herglotz function $M \in \mathbb H$, for any Borel subset $S\subset \R$ and for all $j =1,2, \dots, n$ we have the following estimate,
 
 \begin{equation} \label{eq 4.22}  \Big | \int_{A_j} \omega_{c^*  M (N,t+iy)c} (-S) dt - \int_{A_j} \omega_{c^*  M (N,t)c}(-S) dt\  \Big | \leq \epsilon |A_j| .\end{equation}
We can define $y$ for each value of $j$ and $c\in \mathbb B(0)$; so $y$ is a function of $j $ and $c$, $ y =y(j,c)$. However, by taking the infimum over the sets $ \{1,2, \dots, n\}$ and $\mathbb B(0)$, we can assume $y$ to be independent of $j$ and $c$.

Since $P_+(n,t) \in  \operatorname{Aut}(\mathcal S_d)$ we obtain from \eqref{eq 4.14} and \eqref{eq 4.16}
that
 \begin{equation} \label{eq 4.18} d_{\infty}(\tilde M_+(n,t), \ P_+(n,t) M^j) \leq \epsilon \end{equation}
 for all $t \in A_j,\quad  1\leq j \leq N .$

Using \eqref{eq 4.11} and integrating show that for arbitrary Borel set $S \subset \R$ and for any $c \in \mathbb B(0)$
  \begin{align*}  \Big | \int_{A_j} \omega_{c^* \tilde M_+(n,t)c} (S) \ dt -  & \int_{A_j} \omega_{c^*  P_+(n,t)M^jc}(S) \ dt\  \Big | \\   & \leq  \int_{A_j}  \Big | \omega_{c^* \tilde M_+(N,t)c} (S) - \omega_{c^*  P_+(n,t)M^jc}(S) \Big|\ dt \\ & \leq  \gamma ( c^* \tilde M_+(n,t)c, \ c^*P_+(n,t) M^jc ) \ |A_j| \\ & \leq d_{\infty}(\tilde M_+(n,t), \ P_+(n,t) M^j) \ |A_j|. \end{align*}
  Thus by using \eqref {eq 4.18}
   \begin{align} \label{ineq 1}  \Big | \int_{A_j} \omega_{c^* \tilde M_+(N,t)c} (S) \ dt -  & \int_{A_j} \omega_{c^*  P_+(n,t)M^jc}(S) \ dt \Big |  \leq \epsilon |A_j|. \end{align}
 Now using   \eqref{hmp}, we can obtain 
 \begin{equation} \label{omega} \omega_{c^*P_+(n,t)M^jc}(S) = \omega_{- \overline{c^*P_+(n,t)M^jc}}(-S).\end{equation}
 Let $W^j = T_+(1, t)(-M^j).$ Then $W^j \in \mathcal S_d .$  
By Lemma \ref{lemma 4.11 } there exists $n_0 \in \N $ such that
\begin{equation} \label{eq 4.26}  d_{\infty}(\tilde M_-(n,t+iy), \ P_-(n,t+iy)( T_+(1,z)(- M^j ) \leq \epsilon \end{equation}
 for all $ n\geq n_0$, $t \in A_j,\quad  1\leq j \leq N .$
Similarly, again by Lemma \ref{lemma 4.6} and on integration over $A_j$, we estimate
 \begin{equation} \label{eq 4.28} \Big | \int_{A_j} \omega_{c^* P_-(n, t+iy) W^jc}(-S) dt - \int_{A_j} \omega_{c^* \tilde M_-(n, t+iy)c}(-S) dt\  \Big | \leq \epsilon |A_j| ,\end{equation} for any $c \in \mathbb B(0), n\geq n_0.$
Now by triangle inequality we get,
\begin{align*}   &\Big | \int_{A_j} \omega_{c^* P_-(n, t) W^jc}(-S) dt -  \int_{A_j} \omega_{c^* \tilde M_-(n, t)c}(-S) dt\  \Big | 
\\ & \leq  \Big | \int_{A_j} \omega_{c^* P_-(n,t) W^j c}(-S) dt - \int_{A_j} \omega_{c^* \tilde P_-(n, t+iy)W^jc}(-S) dt \Big | \\  & +\Big |\int_{A_j} \omega_{c^* \tilde P_-(n, t+iy)W^jc}(-S) dt - \int_{A_j} \omega_{c^*  \tilde M_-(n, t+iy)c}(-S) dt  \Big | \\ & + \Big |
\int_{A_j} \omega_{c^*  \tilde M_-(n, t+iy)c}(-S) dt  -
\int_{A_j} \omega_{c^* \tilde M_-(n, t)c}(-S) dt\  \Big | .\end{align*} 
Again by  \eqref{eq 4.22}   and \eqref{eq 4.28}, for any $c \in \mathbb B(0),  n\geq n_0$
\begin{align} \label{ineq2}  &\Big | \int_{A_j} \omega_{c^* P_-(n,t) W^j c}(-S) dt -  \int_{A_j} \omega_{c^* \tilde M_-(n,t)c}(-S) dt\  \Big | \leq 3\epsilon |A_j| . 
\end{align} 
  Using $ P_-(n,t) W^j = - \overline {P_+(n,t)M^j}$, for all for all $c\in \mathbb B(0)$ if $n\geq n_0 $ we get
 
  \begin{align*} 
  &\Big | \int_{A_j} \omega_{c^* \tilde M_-(N,t)c} (-S) dt   -\int_{A_j} \omega_{c^* \tilde M_+(N,t)c}(S) dt\ \Big | \\ 
  & \leq 
  \Big | \int_{A_j} \omega_{c^* \tilde M_-(N,t)c} (-S) dt  - \int_{A_j} \omega_{c^* P_-(n,t) W^j c}(-S) dt \Big |\\ & + \Big | \int_{A_j} \omega_{-c^*  \overline {P_+(n,t)M^j }c }(-S) dt - \int_{A_j} \omega_{c^* \tilde M_+(N,t)c}(S) dt\ \Big |
  .\end{align*}  
 By using \eqref{omega} and   combining this with \eqref{ineq2}, \eqref{eq 4.26}, \eqref{eq 4.28} we have,
 
 \begin{align*} \Big | \int_{A_j} \omega_{c^* \tilde M_-(N,t)c} (-S) dt   -\int_{A_j} \omega_{c^* \tilde M_+(N,t)c}(S) dt\ \Big |  
   \leq 4\epsilon |A_j| + \epsilon
 . \end{align*}  
 Now taking the sum over $j=1,2, \dots,N$ and using $|A_0 |\leq \epsilon $ we get 
  
  \begin{align*}
  \Big | \int_{A} \omega_{c^* \tilde M_-(N,t)c} (-S) dt   -\int_{A} \omega_{c^* \tilde M_+(N,t)c}(S) dt\ \Big |  
   \leq 4\epsilon |A| + 2\epsilon.
  \end{align*}  \end{proof}
\begin{proof}[Proof of Theorem \ref{RT} ] Let $ \mathbf A \in \omega (\mathbf B)  .$ Then there exists a sequence $n_j \rightarrow \infty $ such that $d(S^{n_j} \mathbf B, \  \mathbf A )  \rightarrow 0.$ By Proposition \ref{prop 4.1}  we get
\[ M_{\pm}(n_j, z) \rightarrow M_{\pm}^ \mathbf A (0,z) = M_{\pm}(z) , \] uniformly on compact subsets of $\C^+$. Then by Theorem \ref{thm 4.6}
\[ 
M_{\pm}(n_j, z) \rightarrow M_{\pm}, \ \  j \rightarrow \infty 
\]
in value distribution, that is, for any $ c\in \mathbb B(0)$ 
\begin{equation} 
\lim_{j \rightarrow \infty }  \int_A \omega_{c^*M_{\pm }(n_j, t)c}(S) dt = \int_A \omega_{c^* M_{\pm } c}(S) dt ,
\end{equation} 
for all Borel set $A, S \subset \R ,\ \ |A| < \infty .$ By Theorem \ref{theorem 4.8} we obtain, for any $A \subset \Sigma_{ac}, \ |A| < \infty $ and $S \subset \R,$ we have

\begin{equation*}
\lim_{j \rightarrow \infty}\Big ( \int_A \omega_{c^* \tilde M_-(n_j,t)c} (-S) dt -\int_A \omega_{c^* \tilde M_+(n_j,t)c}(S) dt\ \Big ) = 0,
\end{equation*} 
for all $c\in \mathbb B(0).$ Then,  it follows that for any $A \subset \Sigma_{ac}, \ |A| < \infty $ and $S \subset \R,$ we have

\begin{equation*}  \Big ( \int_A \omega_{c^*   M_-(t)c} (-S) dt -\int_A \omega_{c^*   M_+(t)c}(S) dt\ \Big ) = 0,
\end{equation*} for all $c\in \mathbb B(0).$
Now by Legesgue differentiation theorem, we get
$$
\omega_{c^*   M_-(t)c} (-S) =  \omega_{c^*   M_+(t)c}(S) , 
$$ for almost every $t\in \Sigma_{ac}$ and for for all $ c\in \mathbb B(0).$ Note that 
$$
\omega_{c^*   M_-(t)c} (-S) = \omega_{- \overline { c^*  M_-(t)c} } (S),
$$
for all $c\in \mathbb B(0).$ Thus for any Borel set $S$,  $ \omega_{c^*  M_+(t)c} (S) =  \omega_{- \overline { c^* M_-(t)c} } (S)$ for all $c\in \mathbb B(0).$ It follows that $ M_+(t) = -\overline {  M_-(t) } $ for almost every $t\in \Sigma_{ac}^d.$ Therefore,  $ \mathbf A \in \mathcal R ( \Sigma_{ac}^d) .$ This completes the proof.
\end{proof}

\textbf{Acknowledgement:} The author sincerely thanks anonymous reviewers for their valuable comments and constructive suggestions, which have greatly improved the article. 

  \providecommand{\bysame}{\leavevmode\hbox to3em{\hrulefill}\thinspace}
\providecommand{\MR}{\relax\ifhmode\unskip\space\fi MR }
\providecommand{\MRhref}[2]{%
  \href{http://www.ams.org/mathscinet-getitem?mr=#1}{#2}
}
\providecommand{\href}[2]{#2}

\end{document}